\documentclass[11pt,a4paper]{article}
\usepackage{amsfonts,amssymb,amsmath,amscd,stmaryrd,latexsym,makeidx,theorem,psfrag}
\author{Luca Martinazzi\thanks{This work was supported by the Swiss National Fond Grant no. PBEZP2-129520.}\\ \small{Centro di Ricerca Matematica Ennio De Giorgi} \\
\small{Scuola Normale Superiore, Pisa} \\ \footnotesize{\texttt{luca.martinazzi@sns.it}}}
\title{A note on $n$-axially symmetric harmonic maps from $B^3$ to $S^2$ minimizing the relaxed energy}
\date{November 25, 2010 $\qquad$ Revised version: May 5, 2011}

\newtheorem{trm}{Theorem}
\newtheorem{prop}[trm]{Proposition}

\newtheorem{lemma}[trm]{Lemma}
\newtheorem{defin}{Definition}

\newcommand{\R}[1]{\mathbb{R}^{#1}}

\newcommand{\de}{\partial}
\newcommand{\ve}{\varepsilon}

\newcommand{\M}[1]{\mathcal{#1}}
\newcommand{\mass}{\mathbf{M}}
\newcommand{\cart}{\mathrm{cart}^{2,1}}

\newcommand{\cur}[1]{\left\llbracket{#1}\right\rrbracket}
\newcommand{\res}{\rule[-0pt]{.5pt}{7pt} \rule[-0pt]{6pt}{.5pt}}
\newenvironment{proof}{\noindent\emph{Proof.}}{\phantom{ } \hfill$\square$\medskip}

\DeclareMathOperator{\supp}{supp}

\DeclareMathOperator{\loc}{loc}

\begin{document}
\maketitle

\begin{abstract} For any $n\ge 2$ we provide an explicit example of an $n$-axially symmetric map $u\in H^1(B_2,S^2)\cap C^0(\overline B_2\setminus \overline B_1)$, where $B_r=\{p\in\R{3}:|p|<r\}$, with $\deg u|_{\de B^2}=0$, ``strictly minimizing in $B_1$'' the relaxed Dirichlet energy of Bethuel, Brezis and Coron $$F(u,B_2):=\frac{1}{2}\int_{B_2}|\nabla u|^2dxdydz+4\pi \Sigma(u,B_2),$$
and having $\Sigma(u,B_2)>0$ , $u|_{B_1}\not\equiv const$. Here $\Sigma(u,B_2)$ is (in a generalized sense) the lenght of a minimal connection joining the topological singularities of $u$. By ``strictly minimizing in $B_1$'' we intend that $F(u,B_2)<F(v,B_2)$ for every $v\in H^1(B_2,S^2)$ with $v|_{B_2\setminus B_1}=u|_{B_2\setminus B_1}$ and $v\not\equiv u$.

This result, which we shall also rephrase in terms of Cartesian currents (following Giaquinta, Modica and Sou\v cek) stands in sharp contrast with a results of Hardt, F-H. Lin and Poon for the case $n=1$, and partially answers a long standing question of Giaquinta, Modica and Sou\v cek. In particular it is a first example of a minimizer of the relaxed energy having non-trivial minimal connection. We explain how this relates to the regularity of minimizers of $F$.
\end{abstract}

\section{Introduction}

\subsection*{The relaxed energy of Bethuel, Brezis and Coron}

Consider a map $u\in H^1(B_2,S^2)=\{v\in H^{1}(B_2,\R{3}):|v|=1\,a.e.\}$ such that $u|_{\de B_2}\in C^0(\de B_2,S^2)\cap H^1(\de B_2,S^2)$ and $\deg(u|_{\de B_2})=0$. The relaxed Dirichlet energy of $u$ was introduced by Bethuel, Brezis and Coron \cite{bbc} as
$$F(u,B_2) :=\frac{1}{2}\int_{B_2}|\nabla u|^2 dxdydz+4\pi \Sigma(u,B_2),$$
with
$$\Sigma(u,B_2):=\frac{1}{4\pi}\sup_{\substack {\xi:B_2\to \R{} \\ \|\nabla \xi\|_\infty\le 1} }\bigg\{\int_{B_2} \mathbf{D}(u)\cdot \nabla \xi dxdydz - \int_{\de B_2} \mathbf{D}(u)\cdot \nu \xi d\M{H}^2\bigg\},$$
(here $\nu(p)=\frac{p}{|p|}$ is the outward unit normal to $\de B_2$) and
$$\mathbf{D}(u):=\bigg(u\cdot \frac{\de u}{\de y}\wedge \frac{\de u}{\de z},u\cdot \frac{\de u}{\de z}\wedge \frac{\de u}{\de x},u\cdot \frac{\de u}{\de x}\wedge \frac{\de u}{\de y}\bigg).$$
The term $\Sigma(u,B_2)$ is a generalization of the idea of minimal connection, already studied by Brezis, Coron and Lieb \cite{bcl} in the sense that if $u$ is smooth away from finitely many points $\{P_i, N_i: 1\le i\le k\}\subset B_2$  and for $\ve$ small one has $\deg u|_{\de B_\ve(P_i)}=1$ and $\deg u|_{\de B_\ve(N_i)}=-1$ then
\begin{equation}\label{minconn}
\Sigma(u,B_2)=\min_{\sigma\in S_k}\sum_{i=1}^k |P_i- N_{\sigma(i)}|,\quad S_k:=\big\{\text{Permutations of }\{1,2,\ldots, k\}\big\},
\end{equation}
see also \cite[p. 37-38]{bbc}. 
As proven in \cite[Thms. 2 - 3]{bbc}, $F$ is the relaxation in the sense of Lebesgue of the Dirichlet energy $D(u,B_2):=\frac{1}{2}\int_{B_2} |\nabla u|^2 dxdydz$, i.e. given $u\in H^1(B_2,S^2)$ as above we have
$$F(u,B_2)=\inf\Big\{\liminf_{k\to\infty} D(u_k,B_2): u_k  \rightharpoonup u\text{ in }H^1, \, u_k\in H^1\cap C^0(\overline B_2,S^2),\, u_k|_{\de B_2}= u|_{\de B_2}\Big\}.$$
In particular $F$  is sequentially weakly lower semicontinuous in $H^1(B_2,S^2)$ in the sense that
$$u_k\rightharpoonup u\text{ in }H^1(B_2,S^2)\text{ and }u_k|_{\de B_2}=u|_{\de B_2} \quad \Rightarrow\quad F(u,B_2)\le\liminf_{k\to\infty}F(u_k,B_2).$$
\begin{defin}\label{min1} Given $u\in H^1(B_2,S^2)$ with $u|_{\de B_2}\in H^1\cap C^0(\de B_2,S^2)$ and $\deg u|_{\de B_2}=0$ we say that $u$ minimizes $F$ in $B_1$ if $F(u,B_2)\le F(v,B_2)$ for every $v \in H^1(B_2,S^2)$ with $v=u$ in $B_2\setminus B_1.$
\end{defin}

An immediate consequence of the semicontinuity of $F$ is that given $\varphi \in H^1(B_2,S^2)$ with $\varphi|_{\de B_2}\in H^1\cap C^0(\de B^2,S^2)$ and $\deg \varphi|_{\de B_2}=0$ we can always find a minimizer $u\in H^1(B_2,S^2)$ of $F$ in $B_1$ with $u=\varphi$ in $B_2\setminus B_1$.

Understanding the regularity of such a minimizer is instead a more subtle and \emph{widely open} problem, to which we want to contribute in this paper. Before doing that, we will recall the approach of Giaquinta, Modica and Sou\v cek to the relaxed energy.

\subsection*{The relaxed energy of Giaquinta, Modica and Sou\v cek}

Later  Giaquinta, Modica and Sou\v cek \cite{gmsb} introduced a different way of relaxing the Dirichlet energy, in the context of Cartesian currents.
Given a map $u\in H^{1}(B_2,S^2)$ and a $1$-dimensional integer multiplicity rectifiable current $L$ in $B_2$, we shall say that the current (in $B_2\times S^2\subset \R{6}$) $T:=\M{G}(u)+L\times \cur{S^2}$
is a Cartesian current if
\begin{equation}\label{bordo0}
\partial \M{G}(u)=-\partial L\times \cur{S^2}\quad \text{in } B_2\times S^2, 
\end{equation}
where $\M{G}(u)=\cur{\{(p,u(p))\in B_2\times S^2:p\in B_2\}}$ denotes the $3$-dimensional current given by integration over the graph of $u$, see \cite{GMS2}. Following \cite{gmsa}, \cite{gmsb} and \cite{gms} we call $\cart (B_2,S^2)$ the set of such currents and set for $T$ as above
$$\M{D}(T,B_2):=\frac{1}{2}\int_{B_2}|\nabla u|^2dxdydz+4\pi\mass(L),$$
where $\mass(L)$ denotes the mass of $L$.
As proven in \cite[Theorem 2]{gms}, $\M{D}$ is the relaxed Dirichlet energy, in the sense that if $\varphi\in C^\infty(B_2, S^2)$, $T\in \cart (B_2,S^2)$ and $T\res ((B_2\setminus \overline{B}_1)\times S^2)=\M{G}(\varphi|_{B_2\setminus \overline{B}_1})$, then there exists a sequence of functions $u_k\in C^{\infty}(B_2,S^2)$ with $$u_k=\varphi\text{ in }B_2\setminus \overline{B}_1,\quad  \M{G}(u_k)\rightharpoonup T\text{ weakly as currents},\quad \frac{1}{2}\int_{B_2} |\nabla u_k|^2dx \to \M{D}(T,B_2).$$
Moreover $\M{D}(\cdot,B_2)$ is sequentially lower semicontinuous with respect to the weak convergence of currents in $\cart (B_2,S^2)$.
\begin{defin}\label{min2}
We say that $T\in \cart (B_2, S^2)$ is a minimizer of $\M{D}$ in $B_1$ if $\M{D}(T,B_2)\le \M{D}(\tilde T,B_2)$ for every $\tilde T\in \cart (B_2, S^2)$ such that   $T\res ((B_2\setminus \overline B_1) \times S^2)=\tilde T\res ((B_2\setminus \overline B_1) \times S^2)$.
\end{defin}

Again semicontinuity of $\M{D}$ implies that for any $T\in \cart (B_2, S^2)$ there exists a minimizer $T_0$ of $\M{D}$ in $B_1$ with $T_0\res(B_2\setminus \overline B_1)\times S^2=T\res(B_2\setminus \overline B_1)\times S^2$.

\medskip

The relation between $\M{D}$ and $F$ was studied in \cite{gms}: Given $u\in H^1(B_2,S^2)$ with $u|_{\de B_2}$ smooth and of degree $0$, there exists a $1$-dimensional integer multiplicity rectifiable current $L$ in $B_2$ which minimizes $\mass(L)$ among the i.m. rectifiable currents satisfying \eqref{bordo0} and $(\de L)\res \de B_2=0$. Moreover $\mass (L)=\Sigma(u,B_2)$.
Therefore $F(u,B_2)=\M{D}(\M{G}(u)+L\times \cur{S^2},B_2)$. In this sense, the current $L$ generalizes the notion of minimal connection of Brezis, Coron and Lieb and $\mass(L)$ provides a natural extension of the length of a minimal connection given by \eqref{minconn}.

An important difference between $F$ and $\M{D}$ is that $F(\cdot,B_2)$ depends only on $u$, but the term $\Sigma(u,B_2)$ is \emph{non-local.} The definition of $\M{D}(\cdot, B_2)$ is local instead, but it depends on the couple $(u,L)$ and not on $u$ only. In order to discuss regularity issues, this second definition turns out to be more convenient because regularity is a local notion. On the other hand, the above considerations show that the two approaches are basically equivalent. In particular if $\M{G}(u)+L\times \cur{S^2}\in \cart(B_2,S^2)$ is a minimizer of $\M{D}$ in $B_1$ in the sense of Definition \ref{min2} with $\supp L\Subset B_2$, $u|_{\de B_2}\in H^1\cap C^0(\de B_2,S^2)$ and $\deg u|_{\de B_2}=0$, then $u$ is a minimizer of $F$ in $B_1$ in the sense of Definition \ref{min1} and conversely, if $u$ is a minimizer of $F$ in $B_1$, then $\M{G}(u)+L\times \cur{S^2}$ is a minimizer of $\M{D}$ in $B_1$ if we choose $L$ minimal under conditions \eqref{bordo0} and $(\de L)\res \de B_2=0$. In both cases $F(u,B_2)=\M{D}(\M{G}(u)+L\times \cur{S^2},B_2)$.

\subsection*{The regularity of minimizers and our example}

Remember that Schoen and Uhlenbeck \cite{su} proved that a  map $u\in H^1(B_2,S^2)$ minimizing the Dirichlet energy $D$ in $B_1$ (in the sense of Definition \ref{min1} with $D$ instead of $F$) is smooth in $B_1$ away from a discrete set (see also \cite{lu}). Their result is sharp as shown by Hardt and F-H. Lin \cite{hl}, who constructed minimizers of $D$ with singular sets finite but arbitrarily large. The theorem of Schoen and Uhlenbeck cannot be applied to the present situation since minimizers of $F$ are not necessarily minimizers of the Dirichlet energy. 

Using a monotonicity formula Giaquinta, Modica and Sou\v cek \cite{gms} proved that if $T=\M{G}(u)+L\times \cur{S^2}\in \cart (B_2, S^2)$ is a  minimizer of $\M{D}$ in $B_1$, then the support of $L\res B_1$ has Hausdorff dimension at most $1$. It is easy to see that $u|_{B_1}$ is a \emph{stationary} harmonic map away from $\supp (L\res B_1)$, and from a theorem of Evans \cite{evans} it follows that $u$ is smooth away from a set of dimension at most $1$. While this result is much weaker than the one of Schoen and Uhlenbeck, we remark that to our knowledge no example has been so far provided of a minimizer of $F$ having singularities (contrary to the case of the Dirichlet energy, where we have the examples of \cite{hl}).

\medskip

In fact Hardt, F-H. Lin and Poon \cite{hlp} were able to give a complete regularity theory for the functional $F$ restricted to the class of \emph{axially symmetric} maps.  A map $u\in H^1(B_2,S^2)$ is said to be $n$-axially symmetric (or simply axially symmetric if $n=1$) if
$$u(r,\theta,z)=(\cos(n\theta) \sin(\varphi(r,z)),\sin(n\theta) \sin(\varphi(r,z)),\cos(\varphi(r,z))),$$
where $(r,\theta,z)$ are cylindrical coordinates in $\R{3}$ and $\varphi$ is a function which determines $u$ completely (compare \cite{hkl}).
Similarly an $n$-axially symmetric Cartesian current in $B_2\times S^2$ will be a current of the form $T=\M{G}(u)+L\times \cur{S^2}\in \cart (B_2, S^2)$, where $u$ is $n$-axially symmetric, the support of $L$ is a subset of the $z-axis$ and its multiplicity at each point is an integer multiple of $n$. We shall call $\M{A}^{(n)}(B_2,S^2)$ the set of such currents.

Hardt, Lin and Poon studied the case $n=1$ and proved (among many other things) that any $T=\M{G}(u)+L\times\cur{S^2}\in \M{A}^{(1)}(B_2,S^2)$ minimizing $\M{D}$ in $B_1$ among axially symmetric currents has $L\res B_1=0$ unless $u|_{B_1}\equiv const$, and from this they deduced that the singular set of  $u|_{B_1}$ is a discrete subset of $z-axis\cap B_1$. (This result is sharp in that they also gave examples where the minimizers \emph{must} have singularities, but we remark that these are minimizers among axially symmetric currents and not among all currents.)  Their clever proof strongly relies on a dipole construction  \cite[Lemma 7.1]{hlp}: assuming that $L\res B_1\neq 0$, they can remove a piece of $L\res B_1$, replace it with a ``dipole'' similar to those introduced in \cite{bcl}, and prove that some energy could be saved, contradicting minimality.

\medskip

Both in Giaquinta, Modica and Sou\v cek's and in Hardt, Lin and Poon's regularity results, proving smallness of the vertical part $L\res B_1$ is crucial, and this suggests the following strategy to prove regularity of a minimizer $u$ of $F$ in $B_1$ (in the sense of Definition \ref{min1}):
\begin{enumerate}
\item Fix $L$ ($1$-d i.m. rectifiable current as above) minimal satisfying \eqref{bordo0} and $(\de L)\res \de B_2=0$ and consider $\M{G}(u)+L\times \cur{S^2}$, which is now a minimizer of $\M{D}$ in $B_1$ in the sense of Definition \ref{min2}.
\item Prove that $L\res B_1=0$ using a generalization of the dipole construction of \cite{bcl} and \cite{hlp} as follows. Assume  that $L\res B_1\neq 0$ and for simplicity that $L\res B_1$ contains a straight segment and that $u$ around this segment behaves almost like an $n$-axially symmetric map; then remove a part of this segment and modify $u$ is the spirit of \cite[Lemma 7.1]{hlp} (for instance using the refined dipole construction of \cite{riviere}) reducing the energy but still preserving condition \eqref{bordo0}, contradiction.
\item $L\res B_1=0$ implies that $u$ is stationary in $B_1$, hence Evans' result implies that $u$ is smooth away from a set of $\M{H}^1$-measure $0$.
\item If possible prove even more regularity for $u$.
\end{enumerate}

In this work we show that the above project fundamentally fails at step 2 because a generalization of the dipole construction of \cite{hlp} to the $n$-axially symmetric case with $n\ge 2$ is impossible! This is an immediate consequence of Theorem \ref{trm1} below.
Define for $\alpha>0$
\begin{equation}\label{defT0}
T_0:= \M{G}(u_0)+ L_0\times \cur{S^2}\in \M{A}^{(n)} (B_2, S^2),
\end{equation}
where
\begin{equation}\label{defT0b}
u_0(r,\theta,z):=\Pi^{-1}\big(\alpha r^n(\cos(n \theta),\sin(n\theta))\big)\in C^\infty(\overline B_2,S^2),\quad L_0:=-n\cur{z-axis}\res B_2.
\end{equation}
Here $\Pi:S^2=\{(x,y,z)\in\R{3}:x^2+y^2+z^2=1\}\to\R{2}\cup\{\infty\}$ is the stereographic projection, given by
$$\Pi(x,y,z)=\frac{(x,y)}{1+z},$$
and $\cur{z-axis}$ is the current given by integration along the $z-axis=\{(0,0,z):z\in\R{}\},$
with orientation set up so that, setting $u_\ve:B_2\to S^2$ as
$$u_\ve(r,\theta,z):=\left\{
\begin{array}{ll}
\Pi^{-1}\Big(\alpha r^{n}\big(\cos(n\theta),\sin(n\theta)\big)\Big)& \text{for }r\ge \ve \\
\Pi^{-1}\Big(\alpha\ve^{2n} r^{-n}\big(\cos(n\theta),\sin(n\theta)\big)\Big)& \text{for }r\le \ve,\rule{0cm}{.7cm}
\end{array}
\right.$$
one has $\M{G}(u_\ve)\in \M{A}^{(n)}(B_2,S^2)$ and as $\ve\downarrow 0$ we have $\M{G}(u_\ve)\rightharpoonup T_0$ as currents in $B_2\times S^2$.

\begin{trm}\label{trm1} For any $n\geq 2$ there is $0<\alpha_0\le \tfrac{1}{4}$ such that for all $\alpha\in [0,\alpha_0]$ the current $T_0$ defined in \eqref{defT0}-\eqref{defT0b} is the unique minimizer of $\M{D}(\cdot,B_2)$ in
$$\M{A}^{(n)}_{T_0}:=\big\{T\in \M{A}^{(n)}(B_2, S^2 ) : T\res ((B_2\backslash \overline{B_1})\times S^2)=T_0\res ((B_2\backslash \overline{B_1})\times S^2) \big\}.$$
\end{trm}

Lemma 7.1 of \cite{hlp} implies at once that for $n=1$ our current $T_0$ is \emph{not} minimizing in $\M{A}^{(1)}_{T_0}$, and if this Lemma could be generalized to the case $n\ge 2$ it would contradict Theorem \ref{trm1}.
The fundamental difference between the cases $n=1$ and $n\ge 2$ is that when $n=1$,  for any minimizer $T=\M{G}(u)+L\times \cur{S^2}\in\M{A}^{(1)}(B_2,S^2)$ of $\M{D}$ in $B_1$ one has that $\widetilde \nabla u:=(\de u/\de x, \de u/\de y)$ cannot vanish indentically on open subsets of the $z-axis\cap B_1$ (\cite[Lemma 7.3]{hlp}), and at points in $\supp L\cap B_1$ with $\widetilde \nabla u\ne 0$ one can remove a piece of $L$ and of the original map and, replacing them with the dipole constructed in \cite[Lemma 7.1]{hlp} (compare with \cite[Section III]{bcl}) one saves an energy proportional to $|\widetilde \nabla u|^2$ (compare also \cite{bc}), hence producing a new current in $\M{A}^{(1)}_T(B_2,S^2)$ with smaller energy, contradicting the minimality of $T$. 
In our example $\widetilde \nabla u_0 \equiv 0$ on the $z$-axis and the expected energy gain due to the dipole replacement is smaller than the energy necessary to glue the dipole to the original map.

Coming back to Step 2 of the regularity program outlined above, if $L\res B_1\neq 0$ contains a segment and $\nabla u$ vanishes along this segment (an occurrence very difficult to rule out in general), then we can expect to be essentially in the situation of Theorem \ref{trm1} and we cannot use minimality to get a contradiction. This remark shows that in order to prove regularity of minimizers of $F$ (or of $\M{D}$) one has to work close to the topological singularities of $u$, i.e. close to $\supp\de L$, and not in the ``interior'' of the minimal connection ($\supp L\setminus \supp \de L$), which might prove very challenging.

\subsection*{Statement of Theorem \ref{trm1} in terms of the $F$ energy}

Theorem \ref{trm1} can be reformulated in terms of the $F$ energy as follows. Define the cones
$$C^+:=\{(r,\theta,z)\in B_2: z>1,\;0\le r<z-1\},\quad C^-:=-C^+=\{p\in \R{3}:-p\in C^+\}$$
and set $\tilde u_0:=u_0$ on $\Omega\setminus(C^+\cup C^-)$, where $u_0$ is as in \eqref{defT0b}. On $C^+$ we define
$$\tilde u_0(r,\theta,z):=\Pi^{-1}\big(\alpha (z-1)^{2n} r^{-n}(\cos(n\theta),\sin(n\theta))\big). $$
On $C^-$ we set $\tilde u_0(r,\theta,z):= \tilde u_0(r,\theta,-z)$. This way
$$\tilde u_0\in H^1(B_2,S^2)\cap C^0(\overline B_2\setminus\{(0,0,\pm 1)\})\cap H^1(\de B_2,S^2)$$
and
$$\deg \tilde u_0|_{\de B_2}=0, \quad\deg \tilde u_0|_{\de B_{1/2}(0,0,\pm 1)}=\mp n$$
(this construction was inspired by the dipole of \cite[Section III]{bcl} and a conversation with H. Brezis).  Theorem \ref{trm1} is essentially equivalent to the following.
\begin{trm}\label{trm2} The map $\tilde u_0$ minimizes $F(\cdot, B_2)$ in the set
$$\M{A}_{\tilde u_0}^{(n)}=\{u\in H^1(B_2,S^2): u\text{ is $n$-axially symmetric and } u=\tilde u_0\text{ in }B_2\setminus B_1\}.$$
\end{trm}
Notice that $\Sigma(\tilde u_0)=2$ (the minimal connection joining the singular points $(0,0,\pm1)$ goes all the way from $(0,0,-1)$ to $(0,0,1)$), while in the case $n=1$ the result of Hardt, Lin and Poon implies that $\tilde u_0$ is not a minimizer and that a minimizer $u$ is smooth in $\overline B_1\setminus \{(0,0,\pm1)\}$ by a simple extention of \cite[Thm. 8.2]{hlp} and $u|_{\overline B_1}$ has singularities at $(0,0,\pm1)$ of degree $\pm1$ which ``topologically'' cancel the singularities of $\tilde u_0|_{B_2\setminus B_1}$ in the sense that (recalling that $u|_{\de B_1}=\tilde u_0|_{\de B_1}\in C^0(\de B_1,S^2)$)
$$\deg u|_{\de (B_{1/2}((0,0,\pm 1))\cap B_1)}=\pm 1,\quad \deg u|_{\de (B_{1/2}((0,0,\pm 1))\setminus B_1)}=\mp 1,\quad \deg u|_{\de B_{1/2}((0,0,\pm 1))}=0$$
and $\Sigma(u)=0$.

\subsection*{Some notation and formulas}

For an open set $\Omega\subset\R{2}$ and a function $u\in W^{1,2}(\Omega,S^2)$ we set $\widehat u=\Pi\circ u$ and we define the Dirichlet energy
\begin{equation}\label{formula1}
E(u,\Omega):=\frac{1}{2}\int_{\Omega} |\nabla u|^2dxdy=2\int_\Omega \frac{|\nabla \widehat u|^2}{(1+|\widehat u|^2)^2}dxdy,
\end{equation} 
and the area counted with multiplicity
\begin{equation}\label{formula2}
A(u,\Omega):=\int_\Omega |Ju|dxdy=4\int_{\Omega}\frac{|J\widehat u|}{(1+|\widehat u|^2)^2}dxdy,
\end{equation} 
where $Ju$ denotes the Jacobian determinant of $u$. Since $|\nabla u|^2\ge 2|Ju|$ one has
\begin{equation}\label{E>A}
E(u,\Omega)\ge A(u,\Omega),
\end{equation}
with equality holding if and only if $u$ is conformal.

Assume now that $\Omega=D_s:=\{(x,y)\in\R{2}: x^2+y^2< s^2\}$
and $u$ is $n$-axially symmetric, i.e. for a function $f:[0,s]\to \overline{\R{}}$ we can write in polar coordinates
\begin{equation}\label{fu}
u(r,\theta)=\Pi^{-1}\big(f(r)(\cos(n\theta), \sin (n\theta))\big).
\end{equation}
Then a simple computation shows
\begin{equation}\label{duJu}
\frac{1}{2}|\nabla\widehat u|^2=\frac{|f'|^2}{2}+\frac{n^2 f^2}{2r^2}\ge  \frac{nf|f'|}{r}=|J\widehat u|.
\end{equation}

\begin{lemma}\label{lemmaarea}
If $f:[s,t]\to [0,\infty]$ is any function with $0\le s<t$, $f(s)=a$, $f(t)=b$ ($\lim_{r\uparrow t}f(r)=\infty$ if $b=\infty$), $a\le b$ and $u\in W^{1,2}(B_t\setminus B_s)$ is as in \eqref{fu}, then 
\begin{equation}\label{contoarea}
A(u,D_t\setminus D_s)\ge 4\pi n\frac{b^2}{1+b^2}-4\pi n\frac{a^2}{1+a^2}, \qquad \Big(\frac{b^2}{1+b^2}=1\text{ if }b=\infty\Big).
\end{equation}
The inequality is an equality if and only if $f$ is monotone. An analogous statement applies when $a>b$ (possibly with $a=\infty$).
\end{lemma}

\begin{proof} Assume first $0< b<\infty$. Then we compute, using \eqref{formula2} and \eqref{duJu},
\begin{equation*}
\begin{split}
A(u,D_t\setminus D_s)&=4\pi n\int_s^t \frac{2 f |f'|}{(1+f^2)^2}dr\ge 4\pi n\int_s^t \frac{2 f f'}{(1+f^2)^2}dr\\
&=4\pi n\bigg[-\frac{1}{1+f^2}\bigg]_s^t=\frac{4\pi n b^2}{1+b^2}-\frac{4\pi n a^2}{1+a^2},
\end{split}
\end{equation*}
where the inequality is strict if and only if $f'<0$ on a set of positive measure, i.e. if $f$ is not monotone. When $b=\infty$ the same proof applies, up to a simple approximation procedure. The case $a>b$ is similar.
\end{proof}

\noindent In the following $C$ will denote a large positive constant which may change from line to line.

\section{Proof of Theorem \ref{trm1}}

Consider the open cylinder $\Sigma:=\{(x,y,z)\in \R{3}:x^2+y^2<1, -1<z<1\}.$
Since $B_1\subset\Sigma\Subset B_2$, it suffices to prove that $T_0$ minimizes
$$\M{D}(T,\Sigma):=\frac{1}{2}\int_{\Sigma}|\nabla u|^2dxdydz +4\pi \mass(L\res \Sigma)$$
over $\M{A}^{(n)}_{T_0,\Sigma}:=\big\{T\in \M{A}^{(n)}(B_2, S^2 ) : T=T_0 \textrm{ in } (B_2\backslash \overline \Sigma)\times S^2\big\}$. This will simplify the notation.

The proof proceeds by contradiction. Let from now on $n\geq 2$ be fixed and let us assume that there exists a current $T=\M{G}(u)+L\times \cur{S^2}\in\M{A}^{(n)}_{T_0,\Sigma}$ with $\M{D}(T,\Sigma)\leq \M{D}(T_0,\Sigma)$ and $T\neq T_0$.
Since $u$ is $n$-axially symmetric, we can find a function $f$ such that
$$u(r,\theta,z)=\Pi^{-1}\big(f(r,z)(\cos(n\theta),\sin(n\theta))\big).$$

\subsection*{Some preliminary lemmas}

\begin{lemma}\label{IT} We have $L=- n\cur{I}$ for some measurable set $I\subset z-axis\cap B_2$.
\end{lemma}
\begin{proof}
The proof is analogous to the one of \cite[Lemma 4.1]{hlp} for the $1$-axially symmetric case, with the following natural modifications.  In Section 2 of \cite{hlp} the $1$-axially symmetric maps $\Lambda(x)=\frac{(x_1,x_2,x_3)}{|x|}$ and $\Psi(x)=\frac{(x_1,x_2,-x_3)}{|x|}$ from $\R{3}\setminus \{0\}$ into $S^2$ should be replaced by the $n$-axially symmetric maps $\Lambda^{(n)}:=R^{(n)}\circ\Lambda$ and $\Psi^{(n)}:= R^{(n)}\circ\Psi$, where $R^{(n)}:S^2\to S^2$ is the map
$$R^{(n)}(\cos\theta\sin\varphi,\sin\theta\sin\varphi,\cos\varphi) = (\cos(n\theta)\sin\varphi,\sin(n\theta)\sin\varphi,\cos\varphi).$$
Notice that $\deg (\pm \Lambda^{(n)}|_{S^2})=\pm n$ and $\deg (\pm \Psi^{(n)}|_{S^2})=\mp n$. With this in mind, the statements and proofs of Lemma 2.1, Lemma 2.2 and Lemma 4.1 of \cite{hlp} can be immediately adapted to the $n$-axially symmetric case.
\end{proof}

Up to modifying $I$ on a set of measure $0$, we can and do assume that every point of $I$ is a Lebesgue point of $I$ with respect to $\M{H}^1\res z-axis$, i.e.
\begin{equation}\label{lebesgue}
\lim_{r\downarrow 0}\frac{\M{H}^1(I\cap B_r(\xi))}{r}=1,\quad \text{for every }\xi\in I.
\end{equation}

\begin{lemma}\label{lemma1} Set $Z:=\overline I\backslash I.$
Then $\M{H}^1(Z)=0$. 
\end{lemma}

\begin{proof} Since $I\cap  (B_2\setminus \overline{\Sigma})= z-axis\cap  (B_2\setminus \overline{\Sigma})$, we have $Z\subset\overline{\Sigma}$. Assume by contradiction that $\M{H}^1(Z)>0$ and let $\xi\in Z\cap \Sigma$ be a Lebesgue point of $Z$ (with respect to $\M{H}^1\res Z$) such that
\begin{equation}\label{eqd}
\lim_{r\to 0}\frac{1}{r}\int_{B_r(\xi)}|\nabla u|^2 dxdydz=0.
\end{equation}
Such a point exists because \eqref{eqd} is true for $\M{H}^1$-almost every $\xi\in z-axis \cap B_2$, by $|\nabla u|^2\in L^1_{\loc}(\R{3})$ and a standard covering argument, see e.g. \cite[Section 2.4.3]{eg}, or \cite[2.10.19(3)]{federer}. Then by the monotonicity argument given in the proof of Theorem 5 of \cite{gms}, one has $\M{H}^1(I\cap B_{r_0}(\xi))=0$ for $r_0>0$ small enough, hence $I\cap B_{r_0}(\xi)=\emptyset$ by \eqref{lebesgue}. This contradicts $\xi\in \overline I$.
\end{proof}

\begin{lemma}\label{lemma2} There is a set $J\subset (z-axis\cap B_2)\backslash \overline I$, such that $\M{H}^1((z-axis\cap B_2) \backslash (\overline I\cup J))=0$ and
\begin{equation}\label{eqz}
\lim_{r\to 0} f(r,z)=+\infty,\quad \text{for }(0,0,z)\in J.
\end{equation}
Similarly
\begin{equation}\label{eqz2}
\lim_{r\to 0} f(r,z)=0,\quad \text{for }\M{H}^1\text{-a.e. }(0,0,z)\in I\cap B_2.
\end{equation}
\end{lemma}

\begin{proof} Since it is obvious that \eqref{eqz2} applies for $(0,0,z)\in B_2\setminus \overline\Sigma$, we will focus on the case $(0,0,z)\in \Sigma$, i.e. $-1<z<1$. We first claim that, for almost every $z\in (-1,1)$, $u\big|_{D_1\times\{z\}}$ is continuous.
Indeed, as shown for instance in \cite[Section 4]{dz} (in the case $n=1$, but the case $n>1$ is identical), $u$ satisfies
\begin{equation}\label{eq}
-\Delta u=|\nabla u|^2 u \quad \textrm{in }\Sigma,
\end{equation}
in the sense of distribution. It is well known, see e.g. \cite[Lemma 3.2.10]{hel}, that the right-hand side of \eqref{eq} belongs to the Hardy space $\M{H}^1_{\loc}(\Sigma)$, hence  $\nabla^2 u\in L^1_{\loc}(\Sigma)$ by elliptic estimates. By a Fubini-type argument, we infer then that
$$(\nabla^2u)\big|_{D_1\times\{z\}}\in L^1_{\loc}(D_1)\quad \textrm{ for almost every } z\in (-1,1),$$
which implies $u\big|_{D_1\times\{z\}}\in C^0_{\loc}(D_1)$, by the embedding $W^{2,1}_{\loc}(D_1)\hookrightarrow C^0_{\loc}(D_1)$. Since $u$ is smooth away from the $z$-axis, see e.g. \cite[Lemma 5.1]{hlp} (where again only the case $n=1$ is treated, but the same proof applies for any $n\ge 1$), we have in fact that $u\big|_{D_1\times\{z\}}\in C^0(D_1)$ for a.e. $z\in (-1,1)$, as claimed.

Let $J\subset (z-axis\cap B_1)\backslash \overline I$ be the set of points $(0,0,z)$ which are  Lebesgue density points of $(z-axis\cap B_1)\backslash \overline I$ (with respect to the $\M{H}^1$ measure), such that $u\big|_{D_1\times\{z\}}\in C^0(D_1)$ and
\begin{equation}\label{slice}
\partial \big(\M{G}\big(u|_{\Sigma(z)}\big)\big) \res \Sigma =\M{G}\big(u|_{D_1\times\{z\}}\big), \quad \Sigma(z):=D_1\times (z,1)\subset\Sigma .
\end{equation}
The slicing property \eqref{slice} is satisfied for almost every $z\in (-1,1)$, since $\M{G}(u)$ is a normal current, see e.g. \cite[Prop. 1, Sec. 2.2.5]{GMS2}, so $\M{H}^1((z-axis\cap B_2)\setminus (\overline I\cup J))=0$.

We now claim that \eqref{eqz} holds true. Fix $z\in (-1,1)$ with $(0,0,z)\in J$. First of all the continuity of $u\big|_{D_1\times\{z\}}$ implies
$$\lim_{r\to 0} f(r,z)=+\infty \quad\textrm{or}\quad \lim_{r\to 0} f(r,z)=0.$$
Since $T$ is a Cartesian current, the degree of the $2$-dimensional current
$$\partial(T \res \Sigma(z))= \de(\M{G}(u\big|_{\Sigma(z)}))+\de( (L\res \Sigma(z))\times \cur{S^2})=\M{G}\big(u\big|_{\partial \Sigma(z)}\big)- n\cur{(0,0,z)}\times \cur{S^2}$$
must be zero (see e.g. \cite[pag. 468]{gmsb}), and this rules out the possibility $\lim_{r\to 0} f(r,z)=0$. This completes the proof of \eqref{eqz}, and the proof of \eqref{eqz2} is completely analogous.
\end{proof}

\subsection*{Strategy of the proof}

Assume first that $L=L_0:=-n\cur{z-axis}\res B_2$. Then $\M{D}(T,\Sigma)\leq \M{D}(T_0,\Sigma)$ is equivalent to
$$\frac{1}{2}\int_{\Sigma}|\nabla u|^2dxdydz\leq \frac{1}{2}\int_{\Sigma}|\nabla u_0|^2dxdydz,$$
and by \eqref{E>A} we have for a.e. $z\in (-1,1)$ that $u\big|_{D_1\times\{z\}}\in W^{1,2}(D_1)\cap C^0(D_1)$ and
\begin{equation}\label{Euu0}
E(u, D_1\times \{z\})\ge A(u, D_1\times \{z\})\ge A(u_0, D_1\times \{z\})=E(u_0, D_1\times \{z\})=4\pi n\frac{\alpha^2}{1+\alpha^2},
\end{equation}
where the first inequality is strict unless $u\big|_{D_1\times\{z\}}$ is conformal by \eqref{E>A}, in the second one we used that $\alpha\in (0,1)$ and Lemma \ref{lemmaarea}, the first equality follows from the conformality of $u_0$, and the second equality follows from Lemma \ref{lemmaarea} and the fact that $f_0(r)=\alpha r^n$ is monotone. Then it easily follows that $u\big|_{D_1\times\{z\}}=u_0\big|_{D_1\times\{z\}}$ for a.e. $z\in (-1,1)$, hence $u=u_0$ and $T=T_0$.
\medskip

Assume now that $L\neq L_0$. Then the set $J$ defined in Lemma \ref{lemma2} has positive $\M{H}^1$-measure. As before, we write $z-axis\cap B_2=\overline I\cup J\cup N$, where $\M{H}^1(N)=0$.
Define for $(0,0,z)\in J$
$$\psi(z):=4\pi n+\frac{4\pi n\alpha^2}{1+\alpha^2}-\frac{1}{2}\int_{D_1\times\{z\}}|\widetilde\nabla u|^2dxdy,\qquad \widetilde \nabla:=\bigg(\frac{\de}{\de x},\frac{\de}{\de y}\bigg).$$
The quantity $\psi(z)$ measures the maximal (because it ignores the $z$-derivative) energy gain (possibly negative) which we can expect by replacing $u_0$ with $u$ in $D_1\times \{ z\}$ and removing the vertical part $n\cur{(0,0,z)}\times \cur{S^2}$.
We must have $\psi(z)>0$ for some $(0,0,z)\in J$, otherwise
\begin{equation*}
\begin{split}
\M{D}(T,\Sigma)>&\int_{(0,0,z)\in I\cap B_1}\bigg(\frac{1}{2}\int_{D_1\times\{z\}}|\widetilde \nabla u|^2 dxdy\bigg)dz+4\pi n\M{H}^1(I\cap B_1) \\
&+\int_{(0,0,z)\in J}\bigg(\frac{1}{2}\int_{D_1\times \{z\}}|\widetilde \nabla u|^2dxdy\bigg)dz\\
\ge &\bigg(\frac{4\pi n \alpha^2}{1+\alpha^2} +4\pi n\bigg)(\M{H}^1(I\cap B_1)+\M{H}^1(J))=\M{D}(T_0,\Sigma),
\end{split}
\end{equation*}
where the first inequality is strict because the integrals on the right don't take into account the $z$-derivative, which cannot vanish identically if $J\neq \emptyset$, and in the second inequality we used \eqref{Euu0} for $(0,0,z)\in I\cap B_1$.
Now we can choose $(0,0,z_1)\in J$ such that
\begin{equation}\label{psi}
\psi(z_1)\ge \frac{1}{2}\sup_{(0,0,z)\in J}\psi(z)>0,
\end{equation}
In the next section we will prove that $\psi(z_1)\le 0$, contradiction.

\subsection*{The energy estimates}

\begin{lemma}\label{lemma0}Let
$$a:=\min_{r\in (0,1]}f(r,z_1) \le \alpha.$$
Then $a>0$,
\begin{equation}\label{stima0}
\psi(z_1)\le \frac{8\pi n a^2}{1+a^2},
\end{equation}
and
\begin{equation}\label{stimadz}
\frac{1}{2}\int_\Sigma \bigg|\frac{\partial u}{\partial z}\bigg|^2dx dy dz\le \frac{32\pi n a^2}{1+a^2}.
\end{equation}
\end{lemma}

\begin{proof} Assume $a\ge  0$ and take any $r\in (0, 1]$ such that $f(r,z_1)=a$. Then Lemma \ref{lemmaarea} and \eqref{eqz} yield
\begin{equation*}
\begin{split}
\frac{1}{2}\int_{D_1\times \{ z_1\}} |\widetilde \nabla u|^2 dxdy \geq & A(u, (D_1\setminus D_r)\times \{z_1\})+ A(u, D_r\times \{z_1\})\\
\ge &\bigg(\frac{4\pi n \alpha^2}{1+\alpha^2}-\frac{4\pi n a^2}{1+a^2}\bigg) +\bigg( 4\pi n -\frac{4\pi na^2}{1+a^2}\bigg),
\end{split}
\end{equation*}
and \eqref{stima0} follows at once. If $a=0$ this yields $\psi(z_1)\le 0$, contradiction. Similarly if $a<0$ choose $0<r_1<r_2<1$ such that $f(r_1,z_1)=f(r_2,z_1)=0$ and $f(r,z_1)\ge 0$ for $r\in (0,r_1)\cup (r_2,1)$, and apply Lemma \ref{lemmaarea} on $(D_1\setminus D_{r_2})\times \{z_1\}$ and on $D_{r_1}\times\{z_1\}$ separately to get again $\psi(z_1)\le 0$.
As for \eqref{stimadz}, for $(0,0,z)\in I$ and $0\le \alpha<1$, \eqref{Euu0} and \eqref{psi} yield
\begin{equation*}
\begin{split}
\frac{1}{2}\int_\Sigma \bigg|\frac{\partial u}{\partial z}\bigg|^2dx dy dz& =
\M{D}(T,\Sigma)-\frac{1}{2}\int_{\Sigma}|\widetilde \nabla u|^2 dxdydz-4\pi n\M{H}^1(I\cap B_1)\\
&\le \M{D}(T_0,\Sigma) -\frac{1}{2}\int_{\Sigma}|\widetilde \nabla u|^2 dxdydz-4\pi n\M{H}^1(I\cap B_1)\\
&\le \int_{(0,0,z)\in J} \psi(z)dz\le 2\psi(z_1)\M{H}^1(J)\le 4\psi(z_1),
\end{split}
\end{equation*}
and the conclusion follows from \eqref{stima0}.
\end{proof}

We have seen that the shape of the profile of $u\big|_{D_1\times\{z_1\}}$, in particular of the infimum of  $f(\cdot, z_1)$, determines the constraint \eqref{stimadz} on the $z$-derivative of $u$. We shall now see how \eqref{stimadz} in turn implies a constraint on the shape of $u$ and consequently a loss of conformality which, for $\alpha$ small enough and $n\ge 2$, forces $\psi(z_1)<0$. This will be the desired contradiction which proves that $L=L_0$ and completes the proof of Theorem \ref{trm1}.

\begin{lemma}\label{lemmasigma} Assume that $0<\alpha\le \tfrac{1}{4}$ and set
$$s:=\inf\big\{r\in (0,1): f(r,z_1)=\tfrac{1}{2}\big\}$$
Then we have $s\leq C_0 a$ for a fixed positive constant $C_0$.
\end{lemma}
\begin{proof}
We have for $0< \alpha\le \tfrac{1}{4}$ and for $r\in (0,s]$
$$f(r,z_1)\geq \frac{1}{2}, \quad f(r,-1)=\alpha r^2\leq \frac{1}{4},$$
hence, by Cauchy-Schwartz's inequality,
$$\int_{-1}^{z_1}\bigg|\frac{\de u(r,z)}{\de z}\bigg|^2dz\geq \frac{1}{z_1+1}\bigg(\int_{-1}^{z_1}\bigg|\frac{\de u(r,z)}{\de z}\bigg|dz \bigg)^2 \geq \frac{1}{z_1+1}|u(r,z_1)-u(r,-1)|^2\geq \frac{1}{C}.$$
Set $\Sigma_s=\{(r,\theta,z)\in \Sigma: r<s\}$. Then
$$\int_{\Sigma_s}\bigg|\frac{\de u}{\de z}\bigg|^2dxdydz\geq \frac{s^2}{C},$$
which together with \eqref{stimadz} implies our claim.
\end{proof}

\begin{prop}\label{lemmakey} For any $n\ge 2$ there is $\alpha_0\in (0,1/4]$ such that if  $0< \alpha\leq \alpha_0$ and $u\in W^{1,2}( D_1, S^2)$ has the form
$$u(r,\theta)=\Pi^{-1}\big(f(r)(\cos( n\theta), \sin(n\theta))\big)$$
with
$$f(1)=\alpha,\quad \lim_{r\to 0}f(r)=+\infty,\quad \min_{0\le r\le 1}f(r)=a,\quad s:= \inf \big\{r\in (0,1]:f(r)=\tfrac{1}{2}\big\}\le C_0a,$$
then
\begin{equation}\label{stimafinale}
\frac{1}{2}\int_{D_1}|\nabla u|^2 dxdy> 4\pi n+\frac{4\pi n \alpha^2}{1+\alpha^2}.
\end{equation}
\end{prop}

Before proving this key proposition, let us notice that it completes the proof of Theorem 1. Indeed we can apply it to $u\big|_{D_1\times\{z_1\}}$ (hence $f(r, z_1)$ will play the role of $f(r)$ in Proposition \ref{lemmakey}) and \eqref{stimafinale} yields $\psi(z_1)<0$.

\medskip

\noindent \emph{Proof of Proposition \ref{lemmakey}.}
In the following several formulas will be more transparent if we write $b$ instead of $1/2$, but the reader should keep in mind that $b$ is fixed. We should also remember that $0< a\le\alpha$ and $\alpha$ is small. Moreover we will often use \eqref{E>A} and Lemma \ref{lemmaarea}.

\medskip

\noindent\emph{Step 1.} We can easily estimate 
\begin{equation}\label{stima5}
E(u,D_s)=\frac{1}{2}\int_{D_s}|\nabla u|^2dxdy\ge  A(u, D_s)=4\pi n -\frac{4\pi n b^2}{1+b^2}.
\end{equation}
To estimate $E(u,D_1\backslash D_s)$ we can assume that $f\le 1$ in $D_1\backslash D_s$. Indeed if $f(r_0,z_1)=1$ for some $r_0\in(s,1)$, we clearly have
\begin{eqnarray*}
E(u,D_1\setminus D_s)&=&E(u,D_{r_0}\backslash D_s)+E(u,D_1\backslash D_{r_0}) \geq A(u,D_{r_0}\backslash D_s)+A(u,D_1\backslash D_{r_0})\\
&\geq&\bigg(2\pi n - \frac{4\pi n b^2}{1+b^2}\bigg) + \bigg(2\pi n- \frac{4\pi n \alpha^2}{1+\alpha^2}\bigg).
\end{eqnarray*}
This and \eqref{stima5} imply \eqref{stimafinale} for $\alpha$ small enough. From now on we shall assume that $f\le 1$ in $D_1\backslash D_s$.

\medskip

\noindent\emph{Step 2.}
Pick any $\tilde s\in (s,1]$ such that $f(\tilde s)=a$. There exists a function $v\in W^{1,2}(D_1\backslash D_s)$
of the form
\begin{equation}\label{vh}
v(r,\theta)=\Pi^{-1}\big(h(r)(\cos( n\theta), \sin(n\theta))\big)
\end{equation}
for some $h\in W^{1,2}([s,1])$ which minimizes the energy
\begin{equation}\label{ev}
E(v,D_1\backslash D_s)=\frac{1}{2}\int_{D_1\backslash D_s}|\nabla v|^2dxdy=4\pi\int_s^1\frac{|h'|^2+\frac{n^2}{r^2}h^2}{(1+h^2)^2} rdr
\end{equation}
among all functions $\tilde v\in W^{1,2}(D_1\backslash D_s)$ (with corresponding $\tilde h\in W^{1,2}([s,1])$ as in \eqref{vh}) satisfying $a\le \tilde h\le 1$, $\tilde h(s)=b$, $\tilde h(\tilde s)=a$ and $\tilde h(1)=\alpha$. Indeed the functional in \eqref{ev} is coercive and the imposed conditions (which are convex) are preserved under the weak convergence in $W^{1,2}$. 

We claim that $h'\le 0$ in $[s,\tilde s]$ and $h'\ge 0$ in $[\tilde s,1]$. Indeed, if for points $s\le s_1<s_2<s_3\le \tilde s$ we have $h(s_1)=h(s_3)<h(s_2)$, we can modify $h$ by setting $h\equiv h(s_1)$ on $[s_1,s_3]$. This would decrease the energy, as one can see by inspecting the right-hand side of \eqref{ev}, using that the function $h\to h^2/(1+h^2)^2$ is strictly increasing for $h\in [0,1]$. One can do the same in $[\tilde s,1]$.

Since $E(u,D_1\backslash D_s)\ge E(v,D_1\backslash D_s)$, it is enough to estimate the energy of $v$. We have
\begin{equation}\label{stima7}
A(v,D_1\backslash D_s)= A(v,D_{\tilde s}\backslash D_s)+ A(v,D_1\backslash D_{\tilde s})= \frac{4\pi nb^2}{1+b^2}-\frac{8\pi n a^2}{1+a^2} +\frac{4\pi n\alpha^2}{1+\alpha^2},
\end{equation}
and the proof is complete if we can prove that for $\alpha$ small enough and $a\in (0,\alpha]$ we have

\begin{equation}\label{stima8}
(E-A)(v, D_1\backslash D_s)>\frac{8\pi n a^2}{1+a^2}.
\end{equation}

\medskip

\noindent\emph{Step 3.} We now reduce the proof of \eqref{stima8} to a simpler problem. From \eqref{formula1}, \eqref{formula2} and  \eqref{duJu} we infer
\begin{equation}\label{E-2A}
\begin{split}
(E-A)(v,D_1\backslash D_s)&=
\int_{D_1\backslash D_s}\frac{\big(2|h'|^2+\frac{2n^2}{r^2}h^2-4\frac{n}{r}|h'|h\big)}{(1+h^2)^2}dxdy\\
&=\int_{D_1\backslash D_s}\frac{2\big(|h'|-\frac{n}{r}h\big)^2}{(1+h^2)^2}dxdy\\
&=4\pi \int_s^1\frac{(|h'(r)|-\frac{n}{r} h(r))^2}{(1+h(r)^2)^2} rdr.
\end{split}
\end{equation}
Since $0\leq h\leq 1$ on $D_1\backslash D_s$, we have $1\leq (1+h^2)^2\leq 4$ in \eqref{E-2A}. Then, considering what we know about $v$ and $h$, to estimate $(E-A)(v,D_1\backslash D_s)$  up to a multiplicative constant it is enough to estimate the infimum of
$$I(g)=\int_s^1\Big(|g'(r)|-\frac{n}{r} g(r)\Big)^2 rdr $$
over
$$\M{C}:=\big\{g\in W^{1,2}([s,1]): g(s)=b,\; g(1)=\alpha,\; g(\tilde s)=a,\; g'\le 0\text{ on }[s,\tilde s],\;g'\ge 0\text{ on }[\tilde s,1]\big\}.$$
Since $I$ is coercive on $\M{C}$ (because $a\le g\le b=1/2$ for  $g\in \M{C}$) and $\M{C}$ is convex and closed with respect to the $W^{1,2}$-topology, it is possible to find a function $g_0$ which minimizes $I$ over $\M{C}$. Since $h\in \M{C}$
\begin{equation}\label{hg0}
(E-A)(v,D_1\backslash D_s)\ge \pi I(h)\ge \pi I(g_0),
\end{equation}
and it remains to estimate $I(g_0)$.

\medskip

\noindent\emph{Step 4.} 
We shall now explicitly compute $g_0$. Consider the set
$$\M{C}_1:=\big\{g\in W^{1,2}([s,\tilde s]): g(s)=b,\; g(\tilde s)=a, \;g'\le 0 \}.$$
Then $g_0 \big|_{[s,\tilde s]}\in \M{C}_1$ and it minimizes
$$\tilde I(g):=\int_s^{\tilde s}\Big( g'(r)+\frac{n}{r}  g(r)\Big)^2 rdr$$
over $\M{C}_1$, where we used that $| g'|=- g'$ for $g\in \M{C}_1$. The functional $\tilde I$ is strictly convex over $\M{C}_1$, hence if we can find a critical point $\tilde g$ of $\tilde I$ in $\M{C}_1$, then it has to be the unique minimizer $g_0\big|_{[s,\tilde s]}$. By a critical point in $\M{C}_1$, we mean a function $\tilde g\in \M{C}_1$ such that
\begin{equation}\label{gtilde}
\frac{d}{d\ve}\tilde I(\tilde g+\ve\varphi)\Big|_{\ve=0^+}:=\lim_{\ve\downarrow 0}\frac{I(\tilde g+\ve\varphi)-I(\tilde g)}{\ve}\ge 0, \quad \text{for any }\varphi:= g-\tilde g,\; g\in \M{C}_1.
\end{equation}
The inequality in \eqref{gtilde} is due to the fact that $\M{C}_1$ is not a vector space and $\tilde g$ might belong to $\de\M{C}_1$.

For $t>s$ to be chosen, consider the function
$$\eta_t(r)= A_tr^n+\frac{B_t}{r^n}, \quad A_t=\frac{at^n-bs^n}{t^{2n}-s^{2n}},\quad B_t=\frac{s^n t^n(b t^n-a s^n)}{t^{2n}-s^{2n}},$$
which satisfies $\eta_t(s)=b$, $\eta_t(t)=a$.
There is exactly one value $t_0>s$ for which $\eta'_{t_0}(t_0)=0$. Indeed any such $t_0$ satisfies
\begin{equation}\label{t02n}
t_0^{2n}=\frac{B_{t_0}}{A_{t_0}}=\frac{s^nt_0^n(bt_0^n-as^n)}{at_0^n-bs^n}\quad \text{if }at_0^n-bs^n>0,
\end{equation}
hence
\begin{equation}\label{t0b}
at_0^{2n}-2bs^nt_0^n+as^{2n}=0.
\end{equation}
Then we compute
$$t_{0\pm}^n=\bigg(\frac{b}{a}\pm\sqrt{\Big(\frac{b^2}{a^2}-1\Big)}\bigg)s^n=\frac{b}{a}\bigg(1\pm\sqrt{\Big(1-\frac{a^2}{b^2}\Big)}\bigg)s^n. $$
Then, since we want $t_0>s$, we have
\begin{equation}\label{t0}
t_0^n=t_{0+}^n=\frac{b}{a}\bigg(1+\sqrt{\Big(1-\frac{a^2}{b^2}\Big)}\bigg)s^n=\frac{b}{a}\bigg(2-\frac{1}{2}\frac{a^2}{b^2}+o(a^2/b^2))\bigg)s^n,
\end{equation}
with $\frac{o(a^2/b^2)}{a^2/b^2}\to 0$ as $a/b\to 0$. This way also the condition $at_0^n-bs^n>0$ in \eqref{t02n} is satisfied.

If $t_0\ge \tilde s$ set $\tilde g=\eta_{\tilde s}$. Then $\eta_{\tilde s}'\le 0$ on $[s,\tilde s]$. Indeed $\eta'_{\tilde s}(\tilde s)\le 0$, since this is equivalent to $a\tilde s^{2n}-2bs^n\tilde s^n+as^{2n}\le 0$, which follows from \eqref{t0b} and $t_{0-}\le s< \tilde s\le t_{0+}$. But $\eta_{\tilde s}'(r)\le 0$ is equivalent to $r^{2n}\le B_{\tilde s}/A_{\tilde s}$ and we have proven this for $r=\tilde s$, hence it also holds for $0<r<\tilde s$. 

If $t_0< \tilde s$, set $\tilde g=\eta_{t_0}$ on $[s,t_0]$ and $\tilde g\equiv a$ on $[t_0,\tilde s]$. Again it is clear that $\tilde g'\le 0$.

In both cases we have $\tilde g\in \M{C}_1$ and we claim that $\tilde g$ satisfies \eqref{gtilde}.
In fact, assuming first $t_0<\tilde s$, we have for $\varphi$ as in \eqref{gtilde}
\begin{equation*}
\begin{split}
\frac{d}{d\ve}\tilde I(\tilde g+\ve\varphi)\Big|_{\ve=0^+}&=2\int_{s}^{\tilde s}\Big(\tilde g' +\frac{n}{r}\tilde g \Big)\Big(\varphi'+\frac{n}{r}\varphi\Big) rdr\\ 
&=2\int_s^{t_0}\Big(-(r\tilde g')' +\frac{n^2}{r}\tilde g\Big)\varphi dr +2\int_{t_0}^{\tilde s}\frac{n^2}{r}\tilde g\varphi dr,
\end{split}
\end{equation*}
where we used the condition $\tilde g'(t_0)=0$ in the integration by parts. The last integral is non-negative since $\varphi\ge 0$ in $[t_0,\tilde s]$, being $\tilde g=a$ and $g\ge a$ in that interval. As for the first integral on the right-hand side, it vanishes, since for $t>0$
\begin{equation}\label{el}
-(r \eta_t'(r))' +\frac{n^2}{r}\eta_t(r)=0\quad \text{for }r\in (0,\infty).
\end{equation}
If $t_0\ge \tilde s$, \eqref{gtilde} follows at once from \eqref{el}. Then \eqref{gtilde} is proven and $\tilde g=g_0\big|_{[s,\tilde s]}$.

An analogous procedure can be done on $[\tilde s,1]$, assuming $\tilde s<1$ (if $\tilde s=1$, then $a=\alpha$ and, setting $\tau_0=1$, one has $g_0 \equiv a=\alpha$ on $[t_0,\tau_0]$; then jump to Step 5) and minimizing
$$\bar I( g):=\int_{\tilde s}^1\Big( g'(r)-\frac{n}{r}  g(r)\Big)^2 rdr$$
over
$$\M{C}_2:=\big\{g\in W^{1,2}([\tilde s,1]): g(\tilde s)=a,\; g(1)=\alpha, \;g'\ge 0 \}.$$
We consider for $0<\tau< 1$
$$\zeta_\tau(r)= A_\tau'r^n+\frac{B_\tau'}{r^n},\quad A_\tau'=\frac{\alpha - a \tau^n}{1-\tau^{2n}},\;B_\tau'=\frac{\tau^n(a-\alpha\tau^n)}{1-\tau^{2n}},$$
so that $\zeta_\tau (\tau)=a$, $\zeta_\tau(1)=\alpha$, and we compute $\tau_0\le 1$ such that  $\zeta'_{\tau_0}(\tau_0)= 0$. This gives $\tau_0^{2n}= \frac{B_{\tau_0}'}{A_{\tau_0}'}$, hence
\begin{equation}\label{tau}
\tau_{0\pm}^{n}=\frac{\alpha}{a}\bigg(1\pm\sqrt{1-\frac{a^2}{\alpha^2}}\bigg),\quad \tau_0^n=\tau_{0-}^n=\frac{\alpha}{a}\bigg(1-\sqrt{1-\frac{a^2}{\alpha^2}}\bigg), 
\end{equation}
where we chose the minus sign because $\tau_0 \le 1$ (simple algebraic computations show that $\tau_{0-}\le 1$, with equality if and only if $a=\alpha$).
As before if $\tau_0\le \tilde s$ we set $\bar g=\zeta_{\tilde s}$, if $\tilde s <\tau_0<1$ we set  $\bar g=\zeta_{\tau_0}$ on $[\tau_0,1]$ and $\bar g\equiv a$ on $[\tilde s, \tau_0]$, if $\tau_0=1$ we set $\bar g\equiv a=\alpha$ on $[\tilde s,1]$. Then again $\bar g$ minimizes $\bar I$ over $\M{C}_2$, hence $\bar g=g_0\big|_{[\tilde s, 1]}$.

\medskip

\noindent\emph{Step 5.} We have completely determined $g_0$ (depending on $a,\alpha, s$ and $\tilde s$ only). In particular we have proven that $g_0\equiv a$ on $[t_0,\tau_0]$.

We now prove that $t_0/\tau_0\to 0$ as $\alpha\to 0$ and complete the proof of \eqref{stima8}. First of all notice that \eqref{t0} and Lemma \ref{lemmasigma} imply (keeping in mind that $b=1/2$)
\begin{equation}\label{t02}
t_0^n\le Ca^{n-1}.
\end{equation}
To estimate $\tau_0$ we go back to \eqref{tau} and write $\beta=(a/\alpha)^2\in (0,1]$. We claim that
\begin{equation}\label{tau02}
\tau_0^n=\frac{1}{\sqrt{\beta}}\big(1-\sqrt{1-\beta}\big)\ge \frac{\sqrt{\beta}}{C}=\frac{1}{C}\frac{a}{\alpha},
\end{equation}
where $C$ is fixed. Indeed this reduces to prove that
$$\varphi(\beta):=\frac{1}{\beta}\big(1-\sqrt{1-\beta}\big)\ge \frac{1}{C}\quad \text{for }\beta\in (0,1],$$
which is obvious since $\varphi>0$ in $(0,1]$ and $\lim_{\beta\downarrow 0}\varphi(\beta)=\frac{1}{2}$.
Since $n\ge 2$, from \eqref{t02} and \eqref{tau02} we infer
\begin{equation}\label{n2}
\frac{t_0}{\tau_0}\le C\big(\alpha a^{n-2}\big)^{\frac{1}{n}}\to 0\quad \text{as }\alpha\to 0.
\end{equation}
Then we have with \eqref{hg0}
\begin{equation*}
(E-A)(v,D_1\backslash D_s)\geq \pi I(g_0)\geq \pi\int_{t_0}^{\tau_0}\Big(\frac{na}{r}\Big)^2  r dr \ge \pi n^2 a^2\log\frac{\tau_0}{t_0}=\frac{a^2}{o(1)},
\end{equation*}
with $o(1)\to 0$ as $\alpha\to 0$, and \eqref{stima8} holds true if $0<\alpha\le \alpha_0=\alpha_0(n)$.
\hfill $\square$

\section{Proof of Theorem \ref{trm2}}

Theorem \ref{trm2} can be proven essentially as Theorem \ref{trm1} after fixing a minimal connection. Here instead we show how to deduce it from Theorem \ref{trm1}, to emphasize that the two theorems are equivalent (and similarly one could also deduce Theorem \ref{trm1} from Theorem \ref{trm2}).

\medskip

Let $u$ be a minimizer of $F(\cdot,B_2)$ in $\M{A}^{(n)}_{\tilde u_0}$. It follows from \cite{hlp} that $u|_{\overline B_1}$ is smooth away from the $z$-axis. Now fix $L$ minimizing the $1$-dimensional mass in the set of $1$-dimensional currents satisfying $(\de L)\res \de B_2=0$ and \eqref{bordo0}. Such a minimizer exists because the above set is closed with respect to the weak convergence of currents.  We first claim that $L=\pm n \cur{I}$ for an $\M{H}^1$-measurable set $I\subset z-axis\cap B_1$, so that $T:=\M{G}(u)+L\times \cur{S^2}\in \M{A}^{(n)}(B_2,S^2)$. Indeed it follows from the generalization of Lemma 4.1 of \cite{hlp} to the case $n\ge 2$ (see  the proof of Lemma \ref{IT} above), that $L=\pm n\cur{I}$ for an $\M{H}^1$-measurable set $I\subset z-axis\cap B_2$, but since $u|_{\overline B_2\setminus\overline B_1}\in C^\infty$ and $(\de L)\res \de B_2=0$, it follows from \eqref{bordo0} that $\supp L \cap (\overline B_2\setminus \overline B_1)=\emptyset$ by the constancy theorem. 

Now we prove that $T$ minimizes $\M{D}(\cdot, B_2)$ in $\M{A}^{(n)}_T$.
Define
\begin{equation*}
\begin{split}
\tilde T&=T-  T\res((B_2\setminus \overline B_1)\times S^2)+ T_0\res((B_2\setminus \overline B_1)\times S^2)\\
&=T-  \M{G}(\tilde u_0|_{B_2\setminus \overline B_1})+ T_0\res((B_2\setminus \overline B_1)\times S^2),
\end{split}
\end{equation*}
where $T_0$ is as in Theorem \ref{trm1}. From
\begin{equation*}
\begin{split}
(\de \M{G}(\tilde u_0|_{B_2\setminus \overline B_1}))\res (B_2\times S^2)&=\M{G}({\tilde u_0|_{\de B_1}}) +n(\delta_{(0,0,-1)}-\delta_{(0,0,1)})\times \cur{S^2}\\
&=(\de (T_0\res(B_2\setminus \overline B_1)\times S^2))\res( B_2\times S^2)
\end{split}
\end{equation*}
and $(\de T) \res (B_2\times S^2)=0$, we infer  $(\de \tilde T) \res B_2\times S^2=0$, hence $\tilde T$ belongs to $\M{A}^{(n)}_{T_0}(B_2,S^2)$, since we can write it as $\tilde T=\M{G}(\tilde u)+L\times \cur{S^2}$ with $\tilde u:=u\chi_{B_1}+u_0\chi_{B_2\setminus B_1}\in H^1(B_2),$
and the condition \eqref{bordo0} is satisfied.
Clearly $\tilde T$ minimizes $\M{D}(\cdot,B_2)$ in $\M{A}^{(n)}_{T_0}$. Then by Theorem \ref{trm1} $\tilde T=T_0$, hence $u=\tilde u_0$.
\hfill$\square$

\subsubsection*{Aknowledgements}

I wish to warmly thank Prof. Petru Mironescu for useful remarks on the paper and Prof. Tristan Rivi\`ere for many interesting  discussions on the topic of relaxed energies.

\end{document}